\documentclass[11pt,a4paper,reqno]{amsart} 
\usepackage{amsthm}
\usepackage{fullpage}
\usepackage{amsmath,amssymb,amsfonts,latexsym}
\usepackage{blkarray}
\usepackage{color}
\usepackage{caption}
\usepackage{subcaption}
\usepackage{youngtab}
\usepackage{tikz}
\usepackage{ytableau}
\usetikzlibrary{arrows,shapes}
\usepackage{xifthen}
\usetikzlibrary{plotmarks}

\makeatletter
\@namedef{subjclassname@2020}{\textup{2020} Mathematics Subject Classification}
\makeatother

\newcommand{\sn}{\mathcal{S}_n}
\newcommand{\s}{\mathcal{S}}

\theoremstyle{definition}
\newtheorem{theorem}{Theorem} 
\newtheorem{proposition}[theorem]{Proposition}

\newtheorem{corollary}[theorem]{Corollary}

\newtheorem{definition}[theorem]{Definition}

\newtheorem*{remark*}{Remark}
\newtheorem{example}[theorem]{Example}
\newtheorem*{example*}{Example}

\newcommand{\bump}{\mathsf{bump}}
\newcommand{\bumpseq}{\text{bump sequence}}

\newcommand{\RSK}{\mathsf{RS}}
\newcommand{\SYT}{\mathsf{SYT}}
\newcommand{\Graph}{\mathsf{Graph}}
\newcommand{\prem}{\alpha}

\newcommand{\ds}{\displaystyle}

\newcommand{\Part}{\mathcal{P}}
\newcommand{\unpair}[1]{\langle #1 \rangle}
\newcommand{\shape}{\mathsf{shape}}

\newcommand{\weakbump}{\mathsf{weakbump}}
\newcommand{\run}{\mathsf{run}}
\newcommand{\GETOUT}[1]{}
\newcommand{\floor}[1]{\lfloor #1 \rfloor}
\newcommand{\jafar}[1]{\begin{minipage}{0.47\textwidth}\rule{0pt}{2em}#1\end{minipage}}
\newcommand{\sundial}{\hspace*{2em}}
\definecolor{dgreen}{rgb}{0,0.5,0}
\newcommand{\myt}{^{\mathsf{T}}}

\def\bumpdiagramone{
\begin{tikzpicture}[scale=0.8]
\draw[very thin,color=gray] (-0.1,-0.1) grid (5.3,5.3);
\foreach \Point in {(0.5,4.5), (1.5,0.5), (2.5,2.5), (3.5,1.5), (4.5,3.5)}{
	\node at \Point {\large{\bf 0}};
	}
\foreach \Point in {(1.5,4.5),(3.5,2.5) }{
	\node at \Point {\large{\bf 1}};
	}
\foreach \Point in {(3.5,4.5) }{
	\node at \Point {\large{\bf 2}};
	}
\foreach \x/\y/\l in {0.5/5/0,1.5/1/0,1.5/2/0,1.5/3/0,1.5/4/0,1.5/5/1,2.5/3/0,2.5/4/0,2.5/5/0,3.5/2/0,3.5/3/1,3.5/4/1,3.5/5/2,4.5/4/0,4.5/5/0}{
	\node at (\x,\y) [above=-2pt] {\tiny\textcolor{black}{\l}};
	}
\foreach \x/\y/\l in {2/0.5/0,3/0.5/0,4/0.5/0,5/0.5/0,4/1.5/0,5/1.5/0,3/2.5/0,4/2.5/1,5/2.5/1,5/3.5/0,1/4.5/0,2/4.5/1,3/4.5/1,4/4.5/2,5/4.5/2}{
	\node at (\x,\y) [right=-2pt] {\tiny\textcolor{black}{\l}};
	}
\end{tikzpicture}
}

\def\bumpdiagramtwo{
\begin{tikzpicture}[scale=0.8]
\draw[very thin,color=gray] (-0.1,-0.1) grid (9.3,9.3);
\foreach \Point in {(0.5,3.5), (1.5,6.5), (2.5,4.5), (3.5,2.5), (4.5,7.5), (5.5,1.5), (6.5,5.5), (7.5,8.5), (8.5,0.5)}{
	\node at \Point {\large{\bf 0}};
	}
\foreach \Point in {(2.5,6.5),(3.5,3.5),(5.5,2.5),(6.5,7.5),(8.5,1.5) }{
	\node at \Point {\large{\bf 1}};
	}
\foreach \Point in {(3.5,6.5), (5.5,3.5), (8.5,2.5) }{
	\node at \Point {\large{\bf 2}};
	}
\foreach \Point in {(5.5, 6.5), (8.5,3.5) }{
	\node at \Point {\large{\bf 3}};
	}
\foreach \Point in {(8.5,6.5) }{
	\node at \Point {\large{\bf 4}};
	}
\foreach \x/\y/\l in {0.5/4/0,0.5/5/0,0.5/6/0,0.5/7/0,0.5/8/0,0.5/9/0,1.5/7/0,1.5/8/0,1.5/9/0,2.5/5/0,2.5/6/0,2.5/7/1,2.5/8/1,2.5/9/1,3.5/3/0,3.5/4/1,3.5/5/1,3.5/6/1,3.5/7/2,3.5/8/2,3.5/9/2,4.5/8/0,4.5/9/0,5.5/2/0,5.5/3/1,5.5/4/2,5.5/5/2,5.5/6/2,5.5/7/3,5.5/8/3,5.5/9/3,6.5/6/0,6.5/7/0,6.5/8/1,6.5/9/1,7.5/9/0,8.5/1/0,8.5/2/1,8.5/3/2,8.5/4/3,8.5/5/3,8.5/6/3,8.5/7/4,8.5/8/4,8.5/9/4}{
	\node at (\x,\y) [above=-2pt]{\tiny\textcolor{black}{\l}};
	}
\foreach \x/\y/\l in {9/0.5/0,6/1.5/0,7/1.5/0,8/1.5/0,9/1.5/1,4/2.5/0,5/2.5/0,6/2.5/1,7/2.5/1,8/2.5/1,9/2.5/2,1/3.5/0,2/3.5/0,3/3.5/0,4/3.5/1,5/3.5/1,6/3.5/2,7/3.5/2,8/3.5/2,9/3.5/3,3/4.5/0,4/4.5/0,5/4.5/0,6/4.5/0,7/4.5/0,8/4.5/0,9/4.5/0,7/5.5/0,8/5.5/0,9/5.5/0,2/6.5/0,3/6.5/1,4/6.5/2,5/6.5/2,6/6.5/3,7/6.5/3,8/6.5/3,9/6.5/4,5/7.5/0,6/7.5/0,7/7.5/1,8/7.5/1,9/7.5/1,8/8.5/0,9/8.5/0}{
	\node at (\x,\y) [right=-2pt] {\tiny\textcolor{black}{\l}};
	}
\end{tikzpicture}
}

\def\bscale{0.37}
\def\shadowzero{
\begin{tikzpicture}[scale=\bscale]
\draw[very thin,color=gray] (-0.1,-0.1) grid (9.3,9.3);
\foreach \Point in {(1,4), (2,7), (3,5), (4,3), (5,8), (6,2), (7,6), (8,9), (9,1)}{
	\node at \Point {\textbullet};
	}
\foreach \x in {1,2,3,4,5,6,7,8,9}{
	\node at (\x,-0.5) {\tiny \textcolor{gray}{\x}};
	\node at (-0.5,\x) {\tiny \textcolor{gray}{\x}};
	}
\node[] at (5,-2) {\footnotesize$\Graph(\pi)$};
\end{tikzpicture}
}
\def\shadowone{
\begin{tikzpicture}[scale=\bscale]
\draw[very thin,color=gray] (-0.1,-0.1) grid (9.3,9.3);
\foreach \Point in {(1,4), (2,7), (3,5), (4,3), (5,8), (6,2), (7,6), (8,9), (9,1)}{
	\node at \Point {\textbullet};
	}
\foreach \x in {1,2,3,4,5,6,7,8,9}{
	\node at (\x,-0.5) {\tiny \textcolor{gray}{\x}};
	\node at (-0.5,\x) {\tiny \textcolor{gray}{\x}};
	}
\draw[thick] (1,9.3)--(1,4)--(4,4)--(4,3)--(6,3)--(6,2)--(9,2)--(9,1)--(9.3,1);
\draw[thick] (2,9.3)--(2,7)--(3,7)--(3,5)--(9.3,5);
\draw[thick] (5,9.3)--(5,8)--(7,8)--(7,6)--(9.3,6);
\draw[thick] (8,9.3)--(8,9)--(9.3,9);
\foreach \Point in {(4,4),(6,3),(9,2),(3,7),(7,8)}{
    \node at \Point {$\circ$};
}
\foreach \xlabel/\ylabel in {1/1,2/5,5/6,8/9}{
    \node at (\xlabel,9.7) {\tiny \bf\xlabel}; \node at (9.7,\ylabel) {\tiny \bf\ylabel};
}
\node[] at (5,-2) {\footnotesize$\textrm{1st shadows of }\Graph(\pi)$};
\end{tikzpicture}
}
\def\shadowtwo{
\begin{tikzpicture}[scale=\bscale]
\draw[very thin,color=gray] (-0.1,-0.1) grid (9.3,9.3);
\foreach \Point in {(4,4),(6,3),(9,2),(3,7),(7,8)}{
	\node at \Point {\textbullet};
	}
\foreach \x in {1,2,3,4,5,6,7,8,9}{
	\node at (\x,-0.5) {\tiny \textcolor{gray}{\x}};
	\node at (-0.5,\x) {\tiny \textcolor{gray}{\x}};
	}
\draw[thick] (3,9.3)--(3,7)--(4,7)--(4,4)--(6,4)--(6,3)--(9,3)--(9,2)--(9.3,2);
\draw[thick] (7,9.3)--(7,8)--(9.3,8);
\foreach \Point in {(4,7),(6,4),(9,3)}{
    \node at \Point {$\circ$};
}
\foreach \xlabel/\ylabel in {3/2,7/8}{
    \node at (\xlabel,9.7) {\tiny \bf\xlabel}; \node at (9.7,\ylabel) {\tiny \bf\ylabel};
}
\node[] at (5,-2) {\footnotesize$\textrm{2nd shadows of }\Graph(\pi)$};
\end{tikzpicture}
}
\def\shadowthree{
\begin{tikzpicture}[scale=\bscale]
\draw[very thin,color=gray] (-0.1,-0.1) grid (9.3,9.3);
\foreach \Point in {(4,7),(6,4),(9,3)}{
	\node at \Point {\textbullet};
	}
\foreach \x in {1,2,3,4,5,6,7,8,9}{
	\node at (\x,-0.5) {\tiny \textcolor{gray}{\x}};
	\node at (-0.5,\x) {\tiny \textcolor{gray}{\x}};
	}
\draw[thick] (4,9.3)--(4,7)--(6,7)--(6,4)--(9,4)--(9,3)--(9.3,3);
\foreach \Point in {(6,7),(9,4)}{
    \node at \Point {$\circ$};
}
\foreach \xlabel/\ylabel in {4/3}{
    \node at (\xlabel,9.7) {\tiny \bf\xlabel}; \node at (9.7,\ylabel) {\tiny \bf\ylabel};
}
\node[] at (5,-2) {\footnotesize$\textrm{3rd shadows of }\Graph(\pi)$};
\end{tikzpicture}
}
\def\shadowfour{
\begin{tikzpicture}[scale=\bscale]
\draw[very thin,color=gray] (-0.1,-0.1) grid (9.3,9.3);
\foreach \Point in {(6,7),(9,4)}{
	\node at \Point {\textbullet};
	}
\foreach \x in {1,2,3,4,5,6,7,8,9}{
	\node at (\x,-0.5) {\tiny \textcolor{gray}{\x}};
	\node at (-0.5,\x) {\tiny \textcolor{gray}{\x}};
	}
\draw[thick] (6,9.3)--(6,7)--(9,7)--(9,4)--(9.3,4);
\foreach \Point in {(9,7)}{
    \node at \Point {$\circ$};
}
\foreach \xlabel/\ylabel in {6/4}{
    \node at (\xlabel,9.7) {\tiny \bf\xlabel}; \node at (9.7,\ylabel) {\tiny \bf\ylabel};
}
\node[] at (5,-2) {\footnotesize$\textrm{4th shadows of }\Graph(\pi)$};
\end{tikzpicture}
}
\def\shadowfive{
\begin{tikzpicture}[scale=\bscale]
\draw[very thin,color=gray] (-0.1,-0.1) grid (9.3,9.3);
\foreach \Point in {(9,7)}{
	\node at \Point {\textbullet};
	}
\foreach \x in {1,2,3,4,5,6,7,8,9}{
	\node at (\x,-0.5) {\tiny \textcolor{gray}{\x}};
	\node at (-0.5,\x) {\tiny \textcolor{gray}{\x}};
	}
\draw[thick] (9,9.3)--(9,7)--(9.3,7);
\foreach \Point in {}{
    \node at \Point {$\circ$};
}
\foreach \xlabel/\ylabel in {9/7}{
    \node at (\xlabel,9.7) {\tiny \bf\xlabel}; \node at (9.7,\ylabel) {\tiny \bf\ylabel};
}
\node[] at (5,-2) {\footnotesize$\textrm{5th shadows of }\footnotesize\Graph(\pi)$};
\end{tikzpicture}
}

\title{A bump statistic on permutations resulting from the\\ Robinson-Schensted correspondence}
\author[M. Dukes]{Mark Dukes}
\address{UCD School of Mathematics and Statistics, University College Dublin, Dublin 4, Ireland} \email{mark.dukes@ucd.ie}
\author[A. Mullins]{Andrew Mullins}
\address{UCD School of Mathematics and Statistics, University College Dublin, Dublin 4, Ireland} \email{andrew.mullins@ucdconnect.ie}

\begin{document}
\begingroup
\def\uppercasenonmath#1{} 
\let\MakeUppercase\relax 
\maketitle
\endgroup

\begin{abstract}
In this paper we investigate a permutation statistic that was 
independently\footnotemark\ introduced by Romik in 2005. 
This statistic counts the number of bumps that occur during the execution of the Robinson-Schensted procedure when applied to a given permutation. We provide several interpretations of this bump statistic that include the tableaux shape and also as an extremal problem concerning permutations and increasing subsequences. 
Several aspects of this bump statistic are investigated from both structural and enumerative viewpoints.
\end{abstract}

\footnotetext{
After completing this paper we were made aware of a short paper~\cite{romik} by 
Dan Romik in which he showed that the average value of $\bump(\pi)$ over all $\pi \in \sn$ for $n$ large is approximately $128  n^{3/2} / 27\pi^2 $.}

\section{Introduction}
The Robinson-Schensted correspondence~\cite{rs} is a bijection from permutations to pairs of standard Young tableaux of the same shape.
This correspondence and its generalization, the Robinson-Schensted-Knuth correspondence, have become centrepieces of enumerative and algebraic combinatorics due to their many remarkable properties.
In this paper we introduce and investigate a statistic on permutations that is motivated by an aspect of the Robinson-Schensted (RS) correspondence.

Let $\sn$ be the set of all permutations of the set $\{1,\ldots,n\}$.
A {\it{standard Young tableau}} is a filling of the cells of a Young diagram $\lambda$ with integers $\{1,2,\ldots,n\}$ such that each of those integers is used exactly once and the entries are increasing across rows and down columns, e.g.
\begin{center}
    \scriptsize{
        \begin{ytableau}
            1 & 3 & 4 & 5 \\
            2 & 6 & 8 \\
            7 & 9 
        \end{ytableau}    
    }
\end{center}
In order to define the RS correspondence we must first define the insertion of an entry into an existing standard Young tableau.
Let $P$ be a standard Young tableau and $k$ a positive integer.
Let $P \leftarrow k$ be the outcome of the following procedure that inserts $k$ into $P$.
Starting with the top row of $P$:
\begin{enumerate}
\item[(I1)] If $k$ is greater than or equal to the rightmost entry in the current row, then append $k$ to that row. Should the current row be empty then $k$ becomes the solitary entry in that row.
\item[(I2)] Otherwise, determine the rightmost entry $y$ in the current row such that $k < y $. Insert $k$ into the position occupied by $y$. Set $k$ to be $y$, move down one row and {go to} step (I1). 
We say $k$ has {\it{bumped}} $y$ down a row.
\end{enumerate}
The Robinson-Schensted correspondence is now easily defined in terms of the above insertion operation:
let $\pi=\pi_1\cdots\pi_n \in \sn$ and set $(P_0,Q_0)=(\epsilon,\epsilon)$ to be the pair of empty standard Young tableaux.
\begin{itemize}
\item Given the current pair $(P_i,Q_i)$ we create the next pair $(P_{i+1},Q_{i+1})$ by setting $P_{i+1} = P_i \leftarrow \pi_{i+1}$ and forming $Q_{i+1}$ from $Q_i$ by inserting $i+1$ into the cell that is the difference $P_{i+1}\backslash P_i$.
\item The outcome of repeatedly applying the procedure for $i=0,\ldots,n-1$ will be $(P_n,Q_n)=:(P,Q)$. 
\end{itemize}
We will write $\RSK(\pi) =(P,Q)$.
Notice that, by construction, both $P$ and $Q$ will have the same Young tableau shape that we will call $\shape(\pi)$.

\begin{example}\label{exampleone}
Consider the permutation $\pi=475382691 \in \s_9$.
The above procedure applied to $\pi$ produces:
\begin{equation*}
\footnotesize
\begin{array}{rrcl}
\multicolumn{3}{r}{
	\left(\epsilon,\epsilon\right) 
	~ \to ~ \Yvcentermath1 \left(~ \young(4) ~,~ \young(1) ~ \right) 
	~ \to ~ \Yvcentermath1 \left(~ \young(47) ~,~ \young(12) ~ \right) 
	~ \to ~ \Yvcentermath1 \left(~ \young(45,7) ~,~ \young(12,3) ~ \right) 
	~\to} 
	& \Yvcentermath1 \left(~ \young(35,4,7) ~,~ \young(12,3,4) ~ \right)\\[2em]
&&& \multicolumn{1}{c}{\downarrow} \\[1em]
& \multicolumn{1}{c}{\Yvcentermath1 \left(~ \young(256,38,4,7) ~,~ \young(125,37,4,6) ~ \right)}
& \multicolumn{1}{c}{ \leftarrow ~~~~ \Yvcentermath1 \left(~ \young(258,3,4,7) ~,~ \young(125,3,4,6) ~ \right)
 ~~~~~~~\leftarrow} & \Yvcentermath1 \left(~\young(358,4,7) ~,~ \young(125,3,4) ~ \right)\\[3em]
& \multicolumn{1}{c}{\downarrow} && \\[1em]
& 
	\Yvcentermath1 \left(~ \young(2569,38,4,7) ~,~ \young(1258,37,4,6) ~ \right)
 & \multicolumn{2}{l}{\to ~ \Yvcentermath1 \left(~ \young(1569,28,3,4,7) ~,~ \young(1258,37,4,6,9) ~ \right)   ~~~\,\,=~ (P,Q).}
\end{array}
\end{equation*}
\end{example}
Here we have $\shape(\pi)=(4,2,1,1,1)$. \newline

The RS correspondence is typically presented in the above form and the result of each $(P_i,Q_i)\mapsto (P_{i+1},Q_{i+1})$ step is shown. 
It is less common to present and consider how the tableaux are changing in terms of bumps, or indeed keep a count of their number. 
Considering what happens in Example~\ref{exampleone}, we see that the insertion of $\pi_3=5$ into the tableau $\tiny\young(47)$ causes $7$ to be bumped down one row, so the insertion of $\pi_3=5$ causes one bump.
Following this, the insertion of $\pi_4=3$ into $\tiny \Yvcentermath1 \young(45,7)$ causes $4$ to be bumped down one row. 
As a result of this, $7$ in turn is bumped. So the insertion of $\pi_4=3$ causes two bumps.
Do this for all nine entries of $\pi$ to find the sequence of the number of bumps for $\pi_1,\pi_2,\ldots,\pi_9$ is $(0,0,1,2,0,3,1,0,4)$.

\begin{definition}
Given $\pi \in \sn$, let $\bump(\pi)$ be the number of bumps that occur in the application of the RS correspondence to $\pi$.
\end{definition}

In Example~\ref{exampleone} above we have $\bump(\pi)=11$. 
In the remainder of this paper we will make use of two well-known results that we now recall.

\begin{theorem}[Greene's theorem \protect{\cite[A1.1.1]{stanley2}}]
Suppose $\pi \in \sn$ with $\shape(\pi)=\lambda=(\lambda_1,\lambda_2,\ldots)$ and let $\lambda'$ be the conjugate partition to $\lambda$.
The conjugate partition is defined by $\lambda'=(\lambda'_1,\lambda'_2,\ldots)$ where $\lambda'_i := |\{ j ~:~ \lambda_j \geq i\}|$.
Let $I_k(\pi)$ (resp. $D_k(\pi)$) denote the maximal number of elements in a union of $k$ increasing (resp. decreasing) subsequences of $\pi$. 
Then 
\begin{align}
I_k(\pi) &= \lambda_1+\ldots+\lambda_k, \mbox{ and} \label{greenestheoremone}\\
D_k(\pi) &= \lambda'_1+\ldots+\lambda'_k.     \label{greenestheoremtwo}
\end{align}
\end{theorem}

\begin{theorem}[Hook length formula \protect{\cite[Cor. 7.21.6]{stanley2}}]\label{hooklengthformula}
For a Young diagram $\lambda\vdash n$, let $f^{\lambda}$ be the number of distinct {standard Young tableaux} of shape $\lambda$ 
and let $h(c)$ be the \textit{hook length} of cell $c$ of $\lambda$. This quantity $h(c)$ is one plus the number of cells beneath $c$ in that column plus the number of cells to the right of $c$ in that row. Then:
\begin{align*}
f^{\lambda}=\frac{n!}{\prod_{c \in \lambda} h(c)}.
\end{align*}
\end{theorem}

In this paper we will look at some properties and interpretations of this new $\bump$ statistic.
In Section 2, we provide three different interpretations of $\bump(\pi)$. 
The first of these is a function of the tableau shape $\shape(\pi)$ while the second is in terms of Viennot's geometric construction of the Robinson-Schensted correspondence. 
{The third is in terms of its minimal distance from being the union of a prescribed number of increasing subsequences, where the distance mentioned here is the number of terms that must be removed. We show that $\bump(\pi)$ is intimately related to this extremal problem for permutations.}

In Section 3 we define bump sequences that encode the number of bumps that result from the insertion of each permutation element. We show that the entries of this sequence can also be related to an extremal problem for permutations and also relate this bump sequence to the descent set of the permutation.
In addition to this, we introduce bump diagrams that are an alternative method to determining the bump sequence (and consequently the bump) of a permutation.
These diagrams were motived by {\it{growth diagrams}}.

In Section 4 we introduce the {\it{weakbump sequence}} of a permutation. 
This weakbump sequence is essentially an indicator function on each component of the bump sequence of a permutation. 
We show the weakbump statistic, which is defined as the sum of the entries in the weakbump sequence, has several nice properties: it admits a simple expression in terms of $\shape(\pi)$; it coincides with the bump statistic for permutations that are 321-avoiding; it also equals the run statistic that appeared in the recent paper of Gunawan, Pan, Russell, and Tenner~\cite{tenner}.

In Section 5 we examine the generating function of the bump statistic over both the sets of permutations and the set of Young tableaux and prove several results for these polynomials.
When we restrict the generating function to be over the set of 321-avoiding permutations we can give a precise formula for the polynomials and also prove log-concavity of their coefficients.

\section{The bump statistic from three different viewpoints}

In this section we will give characterisations of the bump statistic of a permutation in three different ways. 
The first of these is, upon consideration, a straightforward proof that the bump of a permutation is a weighted sum of the partition lengths of $\shape(\pi)$.

\subsection{Bump in terms of tableau shape}

\begin{proposition}[Bump in terms of tableau shape]
\label{bumpinshape}Given $\pi \in \sn$, let $(P,Q) = \RSK(\pi)$ and suppose that $\shape(\pi) = \lambda=(\lambda_1, \ldots,\lambda_k) \vdash n$.
Then 
\begin{align*}
    \bump(\pi) = \sum_i (i-1) \lambda_i.    
\end{align*}
\end{proposition}

\begin{proof}
Let $\pi, P, Q$, and $\lambda$ be as stated in the theorem. 
According to rules (I1) and (I2) of the RS correspondence, each value in $P$ was first inserted into the first row and then potentially bumped one row at a time further down the tableau. 
We count the total number of bumps by summing the number of values that were bumped from each row. 
Consider row $i$ of $P$. The number of values that were bumped from this row is precisely the number of cells in $P$ beneath row $i$. 
Letting $\bump(\pi,i)$ denote this number,
i.e. $\bump(\pi,i) = \lambda_{i+1}+\lambda_{i+2}+\ldots+\lambda_k$, 
we have that $$\bump(\pi) = \sum_{i=1}^{k-1} \bump(\pi,i) = \sum_{i =1 }^{k-1} (\lambda_{i+1}+\lambda_{i+2}+\ldots+\lambda_k) = \sum_i (i-1) \lambda_i.$$
\end{proof}

Proposition~\ref{bumpinshape} also appears in the main proof of Romik's paper~\cite{romik}.
After discovering this formula for $\bump(\pi)$, we noticed that this quantity $\sum (i-1)\lambda_i$ appears in several places in the algebraic combinatorics literature (see e.g. {Stanley's book}~\cite[Theorem 7.21.2]{stanley2} and {Bergeron's book}~\cite[Eqn 4.29]{bergeron}).
Unfortunately for us, it only does so since it represents the smallest possible weight of a semi-standard Young tableau whose $i+1$th row consists of all $i$'s. 
In each appearance it essentially tells us the first exponent in some power series that has a non-zero coefficient.

As $\bump(\pi)$ is only a function of the shape $\lambda$ of the associated tableaux $\RSK(\pi)$, we will later in the paper sometimes abuse notation where it does not cause confusion and write $\bump(\lambda)$ for that same value.

\subsection{Bump in terms of Viennot's shadows}
Viennot~\cite{viennot77} gave a beautiful geometric description of the RS correspondence in terms of {\it{shadows}}. 
Let us first explain this construction and then show what aspect of the construction represents the bump statistic. 
Suppose $\pi=\pi_1\ldots\pi_n \in \sn$. 
Consider the collection of points 
$\Graph(\pi) = \{(i,\pi_i) ~:~ i \in [n]\}$ in the plane.
Let us place a solid dot at each of these points. 
From each solid dot, extend lines north and east until these lines encounter other lines, or just pass outside of the $n\times n$ grid.
We call these lines the {\it{1st shadows}}.
Let the first row of $P$ (resp. first row of $Q$) be the $y$-index (resp. $x$-index) of where the 1st shadow lines leave the right hand side (resp. top) of the $n\times n$ grid. 

Next, we insert circles on the intermediate points of the 1st shadows where these {\it{intermediate points}} are those corners that did not have solid dots to begin with.
We repeat the same process for these circles to yield a collection of lines that we call the {\it{2nd shadow}}.
The second rows of $P$ and $Q$ are now read in the same manner by observing where the 2nd shadows exit the $n\times n$ grid.
Repeat this process until it is not possible to form any more shadows. The pair $(P,Q)$ constructed in this way is equal to $\RSK(\pi)$ 
(see e.g. Sagan's book~\cite[Section 3.6]{sagan}).
A benefit of this geometric construction is that it makes it easy to observe the classic result $\RSK(\pi^{-1})=(Q,P)$.

\begin{example}\label{exampleoneshadows}
Figure~\ref{figviennot} illustrates Viennot's construction for the permutation $\pi=475382691$ (of Example~\ref{exampleone}). 
We can read off the pair $(P,Q)$ by listing the pairs of (vertical labels, 
horizontal labels) at each step. In this example, 
for the 1st shadows of $\pi$, the vertical labels are (1,5,6,9) while the horizontal labels are (1,2,5,8). 
These correspond to the first rows of $P$ and $Q$ given in Example~\ref{exampleone}.
\begin{figure}
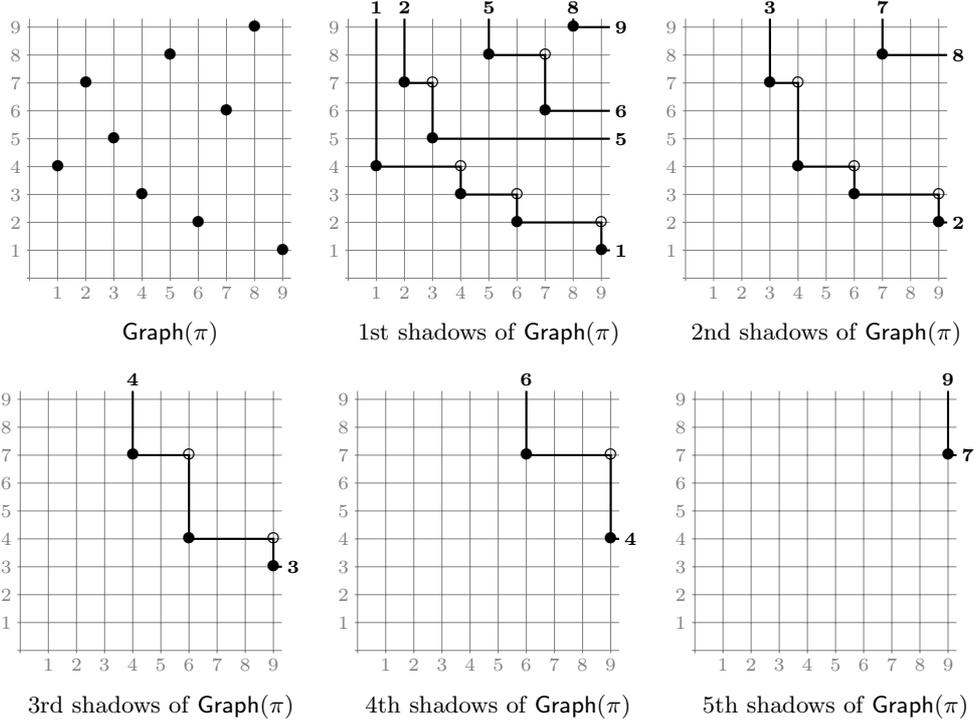

\begin{center}
\shadowzero
\shadowone
\shadowtwo
\end{center}
\begin{center}
\shadowthree
\shadowfour
\shadowfive
\end{center}
\caption{The construction of Viennot's shadows for Example~\ref{exampleoneshadows}.} \label{figviennot}
\end{figure}
\end{example}

\begin{proposition}[Bump in terms of Viennot's shadows]
Let $\pi \in \sn$. 
Consider the construction of the pair $(P,Q)$ using Viennot's method of shadows.
The statistic $\bump(\pi)$ then equals the number of {\it intermediate points} that appear in all of the shadows.
\end{proposition}

\begin{proof}
{We begin with $n$ points on the graph. 
The points with horizontal shadow lines extending east outside of the grid correspond to the first row of $P$. 
Therefore all of the other points correspond to values that are bumped down from the first row of the $P$-tableau to the second row during the RS correspondence. 
The horizontal lines extending from these points end at the \textit{intermediate points}, 
and so the count of the \textit{intermediate points} of the first shadow lines is equal to the count of the number of values bumped out of the first row of the $P$-tableau. 
The same reasoning holds for the second shadow lines (i.e. the count of the number of \textit{intermediate points} of the 2$^\text{nd}$ shadow lines is equal to the number of times values are bumped from the 2$^\text{nd}$ row of the $P$-tableau). 
Continuing this reasoning, we ultimately conclude that the number of \textit{intermediate points} that appear in all of the shadows is equal to the total number of bumps that occur during the RS correspondence, i.e. $\bump(\pi)$. Bergeron discusses this in Section 2.5 of his book \cite{bergeron}.
}
\end{proof}

We know from Proposition~\ref{bumpinshape} that $\bump(\pi)$ depends only on the shape of the tableaux in $\RSK(\pi)$. 
Therefore the observation that $\RSK(\pi^{-1})=(Q,P)$ leads to the conclusion that $\bump(\pi)=\bump(\pi^{-1})$.

\subsection{Bump in terms of a {permutation distance property}}
{In this section we show that the bump statistic on a permutation is intimately related to its minimal distance from being a union of increasing subsequences.}

\begin{theorem}[Bump in terms of a {permutation distance property}]
Let $\pi \in \sn$. 
Let $\prem_i(\pi)$ be the minimal number of elements that one could remove from $\pi$ and it be the case
that what remains is the disjoint union of $i$ increasing subsequences.
Then:
\begin{align}
    \bump(\pi) = \sum_{i\geq 1} \prem_i(\pi)
\end{align}
\end{theorem}

\begin{proof}
Recall from {(\ref{greenestheoremone})} that $I_{k}(\pi)$ denotes the maximal number of elements in a union of $k$ increasing subsequences of $\pi \in \sn$. Then, by (\ref{greenestheoremone}),
$$I_{k}(\pi) = \lambda_{1}+\dots+\lambda_{k}$$ 
As $\pi$ has $n$ elements, by definition of $\prem_i(\pi)$ we therefore have that:
$$\alpha_{i}(\pi) = n - {I_{i}(\pi)} = \sum_{j \geq i+1}{\lambda_{j}}$$
Recalling that $\bump(\pi) = \displaystyle\sum_{j\geq2} (j-1){\lambda_{j}}$ gives:
$$\bump(\pi) = \sum_{i \geq 2}{\sum_{j \geq i}{\lambda_{j}}} = \sum_{i \geq2}{\alpha_{i-1}{(\pi)}} = \sum_{i \geq 1}{\alpha_{i}{(\pi)}}.$$
\end{proof}

\begin{example}
Consider $\pi=475382691$ (of Example~\ref{exampleone}). 
We have that $\prem_1(\pi)$ is the minimal number of elements that one could remove from $\pi$ and it be the case
that what remains is an increasing subsequence. 
There is no increasing subsequence of length 5 in $\pi$, but there are several of length 4. One example is the removal of the elements 7, 3, 2, 6, and 1. Thus $\prem_1(\pi)=5$.

Next, $\prem_2(\pi)$ is the the minimal number of elements that one could remove from $\pi$ and it be the case
that what remains is a union {of} two increasing subsequences. 
There is no single element we can remove from $\pi$ and for what remains to be the union of two increasing subsequences.
Likewise, by inspection there is no pair of elements we can remove from $\pi$ that will result in a union of two increasing subsequence.
However, there are quite a few triples we can remove from $\pi$ and what will be left will be the union of a length-4 with a length-2 increasing subsequence, or the union of two disjoint length-3 increasing subsequences. One example is the removal of the elements 7, 3, and 1.
This gives $\prem_2(\pi) = 3$.

By the same reasoning, we find $\prem_3(\pi)=2$ and $\prem_4(\pi)=1$, and $\prem_5(\pi)=0$ since it is already the union of 5 increasing subsequences.
Therefore $\bump(\pi)=5+3+2+1+0=11$.
\end{example}

Finally, $\prem_i(\pi)$ can be tied back to Viennot's shadow construction by observing that $\prem_i(\pi)$ is the number of intermediate points in the $i$th shadow of $\pi$.

\section{Bump sequences and bump diagrams}
\subsection{Bump sequences} \label{bumpsequences}
Suppose $\pi=\pi_1\ldots\pi_n\in \sn$ and let $(P_j,Q_j)$ correspond to the the $j$th tableaux pair in the construction of $\RSK(\pi)$.
When we construct $(P_{i},Q_{i})$ from $(P_{i-1},Q_{i-1})$ by inserting $P_{i-1} \leftarrow \pi_{i}$, 
what is the number of bumps caused by this insertion in terms of a permutation property of $\pi_1\ldots\pi_{i}$?

\begin{definition} \label{definitionten}
Let $\bump_{i}(\pi)$ be the number of bumps that occur during $P_{i-1} \leftarrow \pi_{i}$. 
We call the sequence $(\bump_{1}(\pi),\ldots,\bump_n(\pi))$ the {\textit{bump sequence}} of $\pi$.
\end{definition}

Notice that the sum $\bump_1(\pi)+\ldots+\bump_n(\pi) = \bump(\pi)$.
First we observe that $\bump_{i}(\pi)$ is equal to one less than the row {index} of $i$ in $Q_i$ and hence also in $Q:=Q_n$.
This means that $\bump_i(\pi)$ can be read from $Q$ after the RS process has been run, if desired. 

\begin{proposition}\label{abcd}
$\bump_{i}(\pi)$ is one less than the smallest value of $j$  such that the  maximal size of the union of $j$ increasing subsequences in $\pi_1\ldots\pi_{i-1}\pi_i$ is greater than that of $\pi_1\ldots\pi_{i-1}$.
\end{proposition}

\begin{proof}
Consider now $\alpha=\pi_1\ldots\pi_{i-1}$ and $\beta=\pi_1\ldots\pi_{i-1}\pi_i$. 
Let {$(P_{i-1},Q_{i-1})$} (resp. {$(P_{i},Q_{i})$}) be the pair of tableaux that correspond to $\alpha$ (resp. $\beta$).
Since $\beta$ is $\alpha$ with a single value appended, we find that the shape of {$P_{i-1}$} differs to the shape of {$P_{i}$} only in that the latter has a cell appended to {an outer corner of $P_{i-1}$}.
Suppose $\lambda$ and $\mu$ are the partitions (i.e. weakly descending sequences of row lengths) corresponding to these shapes, which are almost the same except for one entry.
Then 
\begin{align*}
\bump_{i}(\pi)  
	&= {\min} \{ j: \lambda_j<\mu_j\} -1 \\
	&= {\min} \{ j: \lambda_1+\ldots+\lambda_j<\mu_1+\ldots+\mu_j\} -1 .
\end{align*}
By (\ref{greenestheoremone}), the quantity $\lambda_1+\ldots+\lambda_j$ is the maximal size of the union of $j$ increasing subsequences in $\alpha$, 
and $\mu_1+\ldots+\mu_j$ is the maximal size of the union of $j$ increasing subsequences in $\beta$.
Thus the quantity ${\min} \{ j: \lambda_1+\ldots+\lambda_j<\mu_1+\ldots+\mu_j\}$ is the smallest value of $j$ such that 
the  maximal size of the union of $j$ increasing subsequences in $\beta$ is greater than that of $\alpha$. 
\end{proof}

In particular, if $\bump_i(\pi)=0$ then any longest increasing subsequence of $\pi_1\ldots\pi_i$ is longer than any longest increasing subsequence of $\pi_1\ldots\pi_{i-1}$, and moreover they all end with $\pi_i$.

\begin{proposition}
{The value} $i$ is a {descent} of $\pi$ 
iff $\bump_i(\pi) < \bump_{i+1}(\pi)$.
\end{proposition}

\begin{proof}
Let $\pi \in \sn$ and suppose $(P,Q)=\RSK(\pi)$.
It is a well established fact that an element $i$ is a descent of a permutation $\pi$ iff $i$ is in a row above $i+1$ in the tableau $Q$. 
(See e.g. {Bona's book}~\cite[Theorem 7.15]{bona}.)
From the proof of Proposition~\ref{abcd} above, we noted that the row index of $Q$ that contains $i$ is 
{$1+\bump_i(\pi)$} while the row index of $Q$ that contains $i+1$ is 
{$1+\bump_{i+1}(\pi)$}.
\end{proof}

\begin{corollary}
Let $\pi\in \sn$ and $(P,Q)=\RSK(\pi)$. 
The sequence $(\bump_1(\pi),\bump_2(\pi),\ldots,\bump_n(\pi))$ is the sequence of row indices (less one) of where the numbers $1,2,\ldots,n$ appear in the tableau $Q$.
\end{corollary}

For example, since the $Q$ tableau for the permutation $\pi=475382691$ is
$$\young(1258,37,4,6,9)$$
we have
$$(\bump_1(\pi),\bump_2(\pi),\ldots,\bump_9(\pi)) = (0,0,1,2,0,3,1,0,4).$$

We define a \textit{weak ballot sequence} to be a finite sequence $a_1,a_2,\dots,a_n$ of non-negative integers with the property that: for all non-negative integers $i,j$ with $i<j$ and all $k\leq n$,
$$|\{a_{\ell}  ~:~ a_{\ell}=i \mbox{ and }\ell \leq k \}| \geq |\{ a_{\ell} ~:~ a_{\ell}=j \mbox{ and } \ell \leq k \}|.$$
In other words, if we consider the terms of the sequence to correspond to an ordered list of votes for election candidates then, at all times during the counting of the votes, candidate 0 is (weakly) ahead of candidate 1, who is (weakly) ahead of candidate 2, and so on.
\begin{proposition}
The set of all bump sequences of length $n$ equals the set of weak ballot sequences of length $n$.
\end{proposition}
\begin{proof}
{
First we show that a bump sequence is a weak ballot sequence.
Suppose the bump sequence $(\bump_i(\pi))_{i=1,\ldots,n}$
 associated with some $\pi \in \sn$ is not a weak ballot sequence and let the $k^{\text{th}}$ term be the first term in the sequence that violates the weak ballot sequence property. 
Let $(P,Q)=\RSK(\pi)$.
Recalling that $\bump_i(\pi)$ is equal to one less than the row index of $i$ in $Q$, this implies that in row 
$(\bump_k(\pi)+1)$ of $Q_k$ there must be more cells than in some row of lower index. 
This contradicts the fact that $Q_k$ is a standard Young tableau, which has weakly decreasing row lengths. Thus $(\bump_i(\pi))_{i=1,\ldots,n}$ is a weak ballot sequence.
}

{
Next we show that every weak ballot sequence is the bump sequence of some permutation.
A weak ballot sequence $(a_k)_{k=1,\ldots,n}$ can readily be mapped to a Q-tableau, $Q$ say, as follows: 
starting from the empty tableau, as $i$ ranges from 1 to $n$, append the value $i$ to row $(a_i +1)$ of $Q$. 
This creates a tableau whose row values are strictly increasing by construction. 
At each step of the process the weak ballot sequence property ensures that the row lengths are in weakly descending order, and hence the column values must be strictly increasing. 
So $Q$ is a standard Young tableau. 
Now set $P=Q$ to obtain a valid standard Young tableau pair $(P,Q)$.
As the RS process is reversible 
we can find some $\sigma \in \sn$ with the property $\RSK(\sigma)=(P,Q)$. 
Hence the bump sequence of $\sigma$ is the required weak ballot sequence.
}
\end{proof}

\subsection{Bump diagrams}

Inspired by growth diagrams~\cite{fomingd,krattgd}, in this subsection
we present \textit{bump diagrams}, from which one can read the sequence $(\bump_{i}(\pi))_{i=1}^{n}$, and hence also $\bump(\pi)$. 
For $\pi \in \sn$, begin by constructing an $n\times n$ grid {of squares} and fill the cells $(i,\pi_{i})_{i=1,\ldots,n}$ with $0$s. 
{Note that we are using Cartesian coordinates for cell positions.}
{Next, consider each cell in the first (i.e. bottom) row, moving from left to right. For each cell, execute the five rules (in order) that are given in Figure~\ref{bdr}. Repeat these steps for the second row of the diagram (again moving from left to right). And so on.}

{Once complete, the edges along the top of the diagram will contain the values $(\bump_{i}(\pi))_{i=1}^{n}$ and the edges along the right hand side of the diagram will tell us how many times the value associated with that row gets bumped during the RSK process (put differently, the edge-value plus $1$ of row $i$ is the row of the $P$-tableau containing $\pi_i$).}

{The proof of these statements can be shown by a relatively easy induction on $n$, the length of the permutation. Given a permutation $\pi = \pi_1 \dots \pi_n$, the inductive step involves taking the bump diagram that would be produced by the permutation $\pi' = \pi_1 \dots \pi_{n-1}$ (perhaps doing some relabelling) and then expanding it to include $\pi_n$; it is  not hard to show that the stated rules propagate the desired edge values to the right hand side of the diagram, and that the top of the new column will also contain the desired value.}

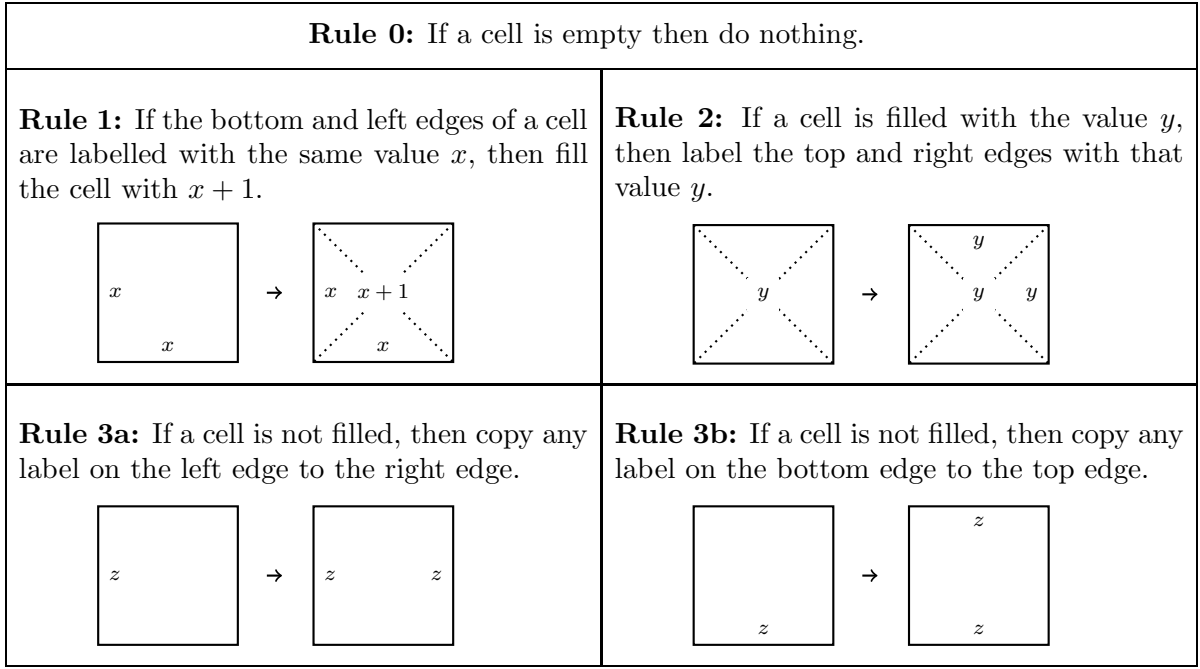
\begin{figure}[!h]
\begin{center}
\begin{tabular}{|c|c|} \hline
\multicolumn{2}{|c|}{\textbf{Rule 0:} If a cell is empty then do nothing.\phantom{$\displaystyle\int$}} \\ \hline
\jafar{\textbf{Rule 1:} If the bottom and left edges of a cell are labelled with the same value $x$, then fill the cell with $x+1$.
\ \\
\sundial\begin{tikzpicture}[font=\scriptsize,thick]
  \foreach \name/\angle in {a/135,b/45,c/315,d/225}
    \node (\name) at (\angle:1.5) {};
    \draw (b.south west) -- (a.south east) --
          node[right=0pt](h){$x$}(d.north east) -- node[above=0pt](v){$x$} (c.north west) -- cycle;
\draw[->] (1.3,0.0) --+ (0.2,0);
\end{tikzpicture}
\begin{tikzpicture}[font=\scriptsize,thick]
  \foreach \name/\angle in {a/135,b/45,c/315,d/225}
    \node (\name) at (\angle:1.5) {};
    \draw (b.south west) -- (a.south east) --
          node[right=0pt](h){$x$}(d.north east) -- node[above=0pt](v){$x$} (c.north west) -- cycle;
    \draw[dotted] (b.south west) -- node[pos=0.5,fill=white]{\textbf{$x+1$}} (d.north east);
    \draw[dotted] (a.south east) -- node[pos=0.5,fill=white]{\textbf{$x+1$}} (c.north west);
\end{tikzpicture}
}
&
\jafar{\textbf{Rule 2:} If a cell is filled with the value $y$, then label the top and right edges with that value $y$.
\ \\
\sundial\begin{tikzpicture}[font=\scriptsize,thick]
  \foreach \name/\angle in {a/135,b/45,c/315,d/225}
    \node (\name) at (\angle:1.5) {};
    \draw (b.south west) -- (a.south east) --
          (d.north east) --  (c.north west) -- cycle;
    \draw[dotted] (b.south west) -- node[pos=0.5,fill=white]{\textbf{$y$}} (d.north east);
    \draw[dotted] (a.south east) -- node[pos=0.5,fill=white]{\textbf{$y$}} (c.north west);
\draw[->] (1.3,0.0) --+ (0.2,0);
\end{tikzpicture}
\begin{tikzpicture}[font=\scriptsize,thick]
  \foreach \name/\angle in {a/135,b/45,c/315,d/225}
    \node (\name) at (\angle:1.5) {};
    \draw  (c.north west) -- node[left=0pt](h){$y$} (b.south west) -- node[below=0pt](v){$y$} (a.south east) --
          (d.north east) --  cycle;
    \draw[dotted] (b.south west) -- node[pos=0.5,fill=white]{$y$} (d.north east);
    \draw[dotted] (a.south east) -- node[pos=0.5,fill=white]{$y$} (c.north west);
\end{tikzpicture}
}
\\ \hline
\jafar{\textbf{Rule 3a:} If a cell is not filled, then copy any label on the left edge to the right edge.
\ \\
\sundial\begin{tikzpicture}[font=\scriptsize,thick]
  \foreach \name/\angle in {a/135,b/45,c/315,d/225}
    \node (\name) at (\angle:1.5) {};
    \draw (b.south west) -- (a.south east) --
          node[right=0pt](h){$z$}(d.north east) -- (c.north west) -- cycle;
\draw[->] (1.3,0.0) --+ (0.2,0);
\end{tikzpicture}
\begin{tikzpicture}[font=\scriptsize,thick]
  \foreach \name/\angle in {a/135,b/45,c/315,d/225}
    \node (\name) at (\angle:1.5) {};
    \draw (c.north west) -- node[left=0pt](h){$z$}(b.south west) -- (a.south east) -- node[right=0pt](h){$z$}(d.north east) -- cycle;
\end{tikzpicture}
}
&
\jafar{\textbf{Rule 3b:} If a cell is not filled, then copy any label on the bottom edge to the top edge.
\ \\
\sundial\begin{tikzpicture}[font=\scriptsize,thick]
  \foreach \name/\angle in {a/135,b/45,c/315,d/225}
    \node (\name) at (\angle:1.5) {};
    \draw (b.south west) -- (a.south east) --
          (d.north east) -- node[above=0pt](v){$z$}(c.north west) -- cycle;
\draw[->] (1.3,0.0) --+ (0.2,0);
\end{tikzpicture}
\begin{tikzpicture}[font=\scriptsize,thick]
  \foreach \name/\angle in {a/135,b/45,c/315,d/225}
    \node (\name) at (\angle:1.5) {};
    \draw (b.south west) -- node[below=0pt](v){$z$}(a.south east) --
          (d.north east) -- node[above=0pt](v){$z$}(c.north west) -- cycle;
\end{tikzpicture}
} \\ \hline
\end{tabular}
\end{center}
\caption{Bump diagram rules. \label{bdr}Note that in the case of the left egde having label $x$ and the bottom edge having label $y$, with $x \neq y$, then both Rules 3a and 3b are to be followed.}
\end{figure}

The bump sequence for $\pi^{-1}$ can also be obtained by reading the right hand side of the diagram, from bottom to top.  This is a consequence of the observation that to plot the diagram of $\pi^{-1}$, one can simply plot $\pi$ and then reflect the points through the south-west, north-east diagonal.

\begin{example}\label{exafone}
For $\sigma = 51324$ we obtain the diagram in Figure~\ref{fig:BumpDiagram_1}.
Reading across the top of the diagram from left to right, we see that the $\bumpseq$ for $\sigma$ is $(0,1,0,2,0)$ and by summing these values we have $\bump(\sigma)=3$.
Moreover, the $\bumpseq$ for $\sigma^{-1}=24351$ is $(0,0,1,0,2)$ which gives $\bump(\sigma^{-1})= 3$.
\end{example}

\begin{example}\label{exaftwo}
For $\pi = 475382691$, the permutation from Example~\ref{exampleone}, we obtain the bump diagram in Figure~\ref{fig:BumpDiagram_2}.
From this we see that the $\bumpseq$ for $\pi$ is $(0,0,1,2,0,3,1,0,4)$ and that $\bump(\pi)=11$.
Moreover, the $\bumpseq$ for $\pi^{-1}=964137258$ is $(0,1,2,3,0,0,4,1,0)$ which gives $\bump(\pi^{-1})=11$.
\end{example}

\begin{figure}[!h]
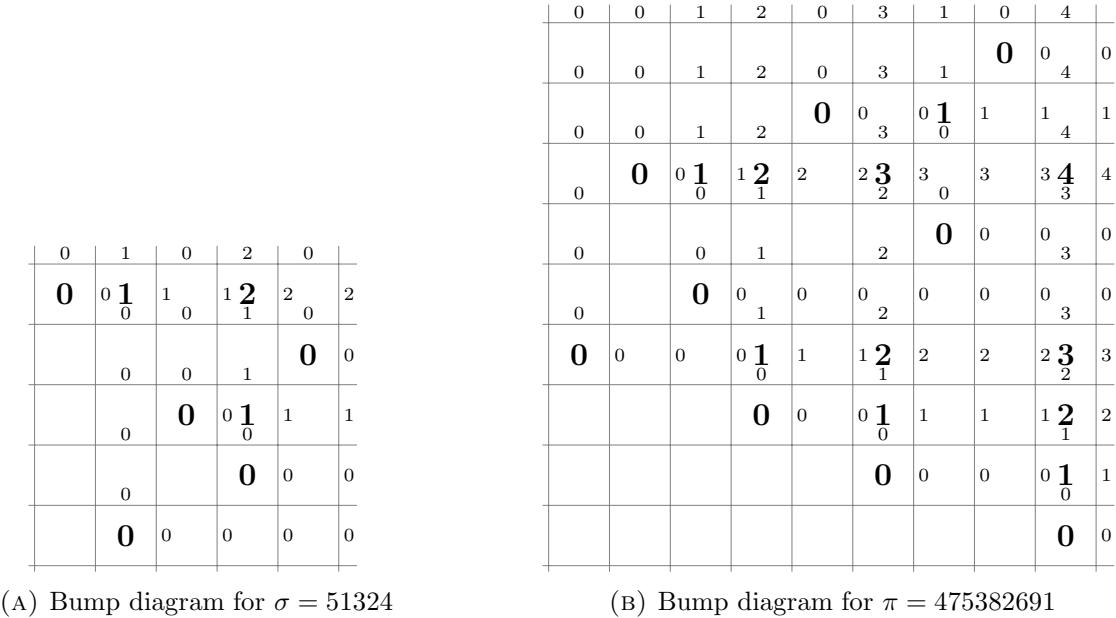

	\centering
	\begin{subfigure}[b]{0.4\textwidth}
		\centering
		\bumpdiagramone
		\caption{Bump diagram for $\sigma=51324$}
		\label{fig:BumpDiagram_1}
	\end{subfigure}
	\hfill
	\begin{subfigure}[b]{0.55\textwidth}
	\centering
	\bumpdiagramtwo
	\caption{Bump diagram for $\pi=475382691$}
	\label{fig:BumpDiagram_2}
	\end{subfigure}
	\caption{The bump diagrams for the permutations in Examples~\ref{exafone} and \ref{exaftwo}}
\end{figure}

Notice also that the bump sequences of $\pi$ and $\pi^{-1}$ contain the same elements (i.e. are equal up to reordering). This is because $\RSK(\pi)$ and $\RSK(\pi^{-1})$ have the same shape.

\section{The weakbump and run permutation statistics}

We turn our attention from counting the number of bumps that occur during the application of RS and focus now on whether or not \textit{any} bumps occur during a given insertion step. 
We find that we can relate this modified statistic to recent research on the RS correspondence.

\begin{definition}
Let $\pi \in \sn$. Define $\weakbump_{i}(\pi)$ to be 1 if $ \bump_{i}(\pi)$ is positive, and 0 otherwise.
Let $$\weakbump(\pi):=\sum_{i=1}^{n} \weakbump_{i}(\pi).$$
\end{definition}

\begin{proposition}\label{weakbumpprop}
Let $\lambda = shape(\pi)$. Then $\weakbump(\pi) = n - \lambda_{1}$.
\end{proposition}

\begin{proof}
We recall steps (I1) and (I2) of the insertion procedure of the RS correspondence. 
A cell is appended to the first row of the tableau $P$ if and only if no bump occurs. 
Therefore the length of the first row, $\lambda_{1}$, is equal to 
$|\{i: \weakbump_i(\pi)=0 \}|$. It follows that $\weakbump(\pi) = n - \lambda_{1}$.
\end{proof}

Recalling (\ref{greenestheoremone}), we can furthermore say that $\weakbump(\pi)$ is therefore equal to 
\textit{the minimum number of terms of $\pi$ that can be {omitted} such that what remains is a strictly increasing sequence}. 

\begin{proposition}
For 321-avoiding permutations, $\weakbump(\pi) = \bump(\pi)$.
\end{proposition}

\begin{proof}
We know from (\ref{greenestheoremtwo}) that 321-avoiding permutations correspond to tableaux having at most two rows. 
From the definition of $\bump_i(\pi)$ ({Definition}~\ref{definitionten}) we can see that, in the 321-avoiding case, $\bump_i(\pi)$ is 1 if it is positive, and 0 otherwise. 
Therefore $\weakbump(\pi) = \bump_i(\pi)$ for all $i$ and so $\bump(\pi) = \weakbump(\pi)$.
\end{proof}

Returning now to the general case of $\pi \in\sn$, there is another nice way to think about $\weakbump(\pi)$. Imagine that the values of $\{1,\ldots,n\}$ have been written on playing cards and that we hold the cards in our hand in the same order as they appear in $\pi$. 
What is the fewest number, $k$, of cards we must rearrange\footnote{{i.e. remove a card from one location and insert into another location.}} in order to sort them into the identity permutation? 
The answer is $n - \lambda_1$, i.e. $\weakbump(\pi)$. 
To see why this is so, consider any longest ascending subsequence of $\pi$. 
Fix the corresponding $\lambda_1$ cards and rearrange the other cards to get the identity permutation. 
This took at most $n - \lambda_{1}$ moves and so $k$ is at most $n-\lambda_1$. 
On the other hand, suppose that $k$ is the smallest number of cards that must be rearranged to give the identity permutation. 
{If we start with $n$ cards, and rearrange $k$ of them in order to achieve the identity permutation, then the $n-k$ cards that were not 
rearranged remain in the same order relative to each other. Therefore they form an ascending subsequence of $\pi$, and so must number at most $\lambda_1$.}
In other words, $k$ must be at least $n-\lambda_1$.

Donald Knuth considers this problem in his book~\cite[Q 5.1.4;41 \textit{Disorder in a library}]{knuth}, where he calls this sorting method a ``deletion-insertion operation''. 
This also appears in Aldous and Diaconis~\cite[Section 1.3]{adia}.

There is a striking correspondence between the statistic $\weakbump$ and \textit{run}, a permutation statistic developed in a recent paper by Gunawan et al.~\cite{tenner}, that we would now like to explain. 
Let $s_{i} \in \sn$ denote the permutation that transposes the quantities at positions $i$ and $i+1$. 
This is known as an \textit{adjacent transposition} and every permutation $\pi \in \sn$ may be written as a product of adjacent transpositions $s_{i_{1}} \dots s_{i_{k}}$. 

Writing $\ell(\pi)$ to denote the minimum number of adjacent transpositions required to express $\pi$, known as the {\it{Coxeter length of $\pi$}},
a \textit{reduced word} for $\pi$ is defined to be any minimal length decomposition into adjacent transpositions $s_{i_{1}} \dots s_{i_{\ell(\pi)}}$. 
A \textit{run} is now defined to be an increasing or decreasing sequence of consecutive integers {in $i_1,i_2,\ldots,i_{\ell(\pi)}$} and $\run(\pi)$ to be the fewest number of runs needed for a reduced word for $\pi$.

In \cite{tenner} the example $\sigma = 314569278$ is used. This has reduced words $[21873456]$ and $[87213456]$. These both consist of three runs, as can be seen by inserting periods for emphasis: $[21.87.3456]$, $[87.21.3456]$. 
It turns out that there is no reduced word of $\sigma$ with fewer runs and so $\run(\sigma)=3$. 

For our example $\pi=475382691$ (from Example~\ref{exampleone}), we find an example of a reduced word is $[67.5.345.23456.12345678]$ and in fact $\run(\pi)=5$. 
In \cite{tenner} it is shown that: $\run(\pi) = n - \lambda_{1}$.
Each run in a reduced word corresponds to the rearrangement of one card in the discussion above and so we have:

\begin{proposition}
Let $\pi \in  \sn $. Then $\weakbump(\pi) = \run(\pi)$.
\end{proposition}

\begin{proof}
As $\weakbump(\pi)$ and $\run(\pi)$ are both equal to $n - \lambda_1$, the result is immediate.
\end{proof}


\section{Bump polynomials}
In this section we will look at two polynomials associated with the bump statistic. 
\subsection{The permutation-bump polynomial}

The first is the generating function of the bump statistic over permutations. Let $b_{n,i}$ be the number of $\pi \in \sn$ with $\bump(\pi)=i$. 
 \begin{definition}[The permutation-bump polynomial]
	Define
    \begin{align}
        B_n(q) =& \sum_{\pi \in \sn} q^{\bump(\pi)} = \sum_i b_{n,i} q^i .
    \end{align}     
 \end{definition}

The first few polynomials are listed in Table~\ref{firstfew}.

\begin{table}[!h]
$$
\boxed{
\footnotesize
\begin{array}{lcl@{\qquad}lcl}
B_1(q) &=& 1  			& T_1(q) &=& 1 \\
B_2(q) &=& 1+q 			& T_2(q) &=& 1+q \\
B_3(q) &=& 1+4q+q^3 		& T_3(q) &=& 1+q+q^3 \\
B_4(q) &=& 1+9q+4q^2+9q^3+q^6 		& T_4(q) &=& 1+q+q^2+q^3+q^6 \\
B_5(q) &=& 1+16q+25q^2+36q^3+25q^4+16q^6+q^{10} & T_5(q) &=& 1+q+q^2+q^3+q^4+q^6+q^{10} 
\end{array}
}$$
\caption{\label{firstfew}}
\end{table}
Proposition~\ref{bumpinshape} shows that the bump statistic depends only on the shape of the tableaux via RS.
\begin{corollary}
    $B_n(q) = \displaystyle\sum_{\lambda \vdash n} (f^{\lambda})^2 q^{\lambda_2+2\lambda_3+3\lambda_4+\ldots}$ where $f^{\lambda}$ is the number of distinct standard Young tableaux of shape $\lambda$.
\end{corollary}
The coefficients for $B_n(q)$ do not enjoy those properties of unimodality/log-concavity that are satisfied by other permutation statistics. 
For example, in
\begin{align*}
B_8(q)=&
1+ 49q+ 400q^2+ 1225q^3+ 4292q^4 + 4900q^5+ 4361q^6+ 9864q^7 + 3136q^8
\\ &
+ 4900q^9+ 1225q^{10}
+ 4096q^{11} +196q^{12}+ 784q^{13}+ 441q^{15}+ 400q^{16}+ 49q^{21}+ q^{28}.
\end{align*}
we can see the polynomial is not symmetric, not unimodal, and there are many `internal' coefficients that are zero.
%
{The reason for the existence of these internal zero coefficients is that, for there to be a non-zero coefficient of an exponent of $q$ in $B_n(q)$, there is a satisfiability problem to be solved for that exponent.}
For example, the coefficient of $q^{25}$ must be zero since it is not possible to find a partition $\lambda=(\lambda_1,\lambda_2,\ldots)\vdash  8$ such that the equation $\lambda_2+2\lambda_3+3\lambda_4+\ldots = 25$ has a solution. \newline

While it is not possible to say much about $B_n(q)$ given its dependence upon $\SYT(\lambda)$, we can say something about the closed form for a special case. 
Consider restricting the sum over all permutations to a sum over all permutations avoiding the pattern 321.

\begin{theorem}
{Let $B_n^{321}(q)$ be the generating function of the bump statistic over the set of length-$n$ 321-avoiding permutations. Then}
$$B_n^{321}(q) =  1+ \displaystyle\sum_{1\leq k \leq n/2} \dfrac{\binom{n}{k,k,n-2k}^2}{\binom{n-k+1}{k}^2} q^k.$$
\end{theorem}

\begin{proof}
Consider a shape $\lambda$ corresponding to a 321-avoiding permutation of length $n$. It can have at most two rows. If it has one row then the contribution to $B_n^{321}(q)$ will be 1 since the only such $\pi$ is the identity permutation, which has bump value zero.

If $\lambda$ has two rows, assume the second row has $k\geq 1$ cells and the first row $n-k$ cells. 
Since the second row can be no longer than the first we must have $k \leq n-k$, i.e. $k\leq n/2$.
The number of standard Young tableaux with shape $\lambda=(n-k,k)$ is given by the hook-length formula (Theorem~\ref{hooklengthformula}):
\begin{align*}
f^{(n-k,k)} &= \dfrac{n!}{k! (n-2k)! (n-2k+2) (n-2k+1) \cdots (n-k+1) }\\
&= \binom{n}{k}\binom{n-k}{k} / \binom{n-k+1}{k} = \binom{n}{k,k,n-2k} / \binom{n-k+1}{k}.
\end{align*}
As $B_n^{321}(q) = 1+\sum_{1\leq k \leq n/2} (f^{(n-k,k)})^2 q^k$, we have the stated expression.
\end{proof}
A sequence $(x_1,\ldots,x_m)$ is called {\it{log-concave}} if $x_i^2\geq x_{i-1}x_{i+1}$ for all $1<i<m$.
\begin{theorem}
The sequence of coefficients of the polynomial $B_n^{321}(q)$ is log-concave.
\end{theorem}

\begin{proof}
Fix $n$ and let 
$a_k$ be the coefficient of $q^k$ in $B_n^{321}(q)$, i.e. 
$$a_k=[q^k] B_n^{321}(q) = {\binom{n}{k,k,n-2k}^2}/{\binom{n-k+1}{k}^2}$$ for all $1\leq k \leq n/2$.
Then $a_k^2 \geq a_{k-1}a_{k+1}$ iff
$$\dfrac{\binom{n}{k,k,n-2k}^4}{\binom{n-k+1}{k}^4} \geq \dfrac{\binom{n}{k-1,k-1,n-2k+2}^2}{\binom{n-k+2}{k-1}^2}  \dfrac{\binom{n}{k+1,k+1,n-2k-2}^2}{\binom{n-k}{k+1}^2}.$$
Since all quantities are strictly positive we see that the above inequality holds true iff it is also true once one takes the positive square roots of both sides:
$$\dfrac{\binom{n}{k,k,n-2k}^2}{\binom{n-k+1}{k}^2} \geq \dfrac{\binom{n}{k-1,k-1,n-2k+2}}{\binom{n-k+2}{k-1}}  \dfrac{\binom{n}{k+1,k+1,n-2k-2}}{\binom{n-k}{k+1}}.$$
On cancelling the many common terms on both sides, this inequality is equivalent to:
$$(k + 1)(n - k + 2)(n - 2k + 1)^2 \geq k(n - k + 1)(n - 2k + 3)(n - 2k - 1).$$
Compare terms on both sides from left to right:
$k+1\geq k$;
$n-k+2\geq n - k + 1$; and 
$(n - 2k + 1)^2\geq (n - 2k + 1)^2-4 =  (n - 2k + 1 +2) (n - 2k + 1 -2) = (n - 2k + 3)(n - 2k - 1)$.
The sequence $(a_k)$ is therefore log-concave.
\end{proof}

Let us now consider a bivariate permutation bump polynomial. 
Let $r(\pi)$ be $\pi$ written in reverse one-line order (for example, if $\pi=25134$ then $r(\pi)=43152$). 
Clearly this reverse operator $r$ is {an involution}.
Define $$B_n(q,t) := \sum_{\pi \in \sn} q^{\bump(\pi)} t^{\bump(r(\pi))}.$$

\begin{proposition} 
$B_n(q,t)=B_n(t,q)$.
\end{proposition}

\begin{proof}
Suppose $\pi\in \sn$ with $\bump(\pi)=a$ and $\bump(r(\pi))=b$. 
Then the permutation $\pi'=r(\pi)$ is such that $\bump(\pi')=b$ and $\bump(r(\pi'))=\bump(\pi)=a$. 
This means to every term $q^a t^b$ in $B_n(q,t)$ there is a corresponding term $q^b t^a$, and so
$B_n(q,t)=B_n(t,q)$.
\end{proof}

We also have the diagonal of this bivariate polynomial:
\begin{proposition}
$B_n(q,q)= 
 \ds\dfrac{1}{q^{n}} \sum_{\lambda \vdash n} (f^{\lambda})^2  \prod_{c \in \lambda} q^{h(c)}.$
\end{proposition}

\begin{proof}
Let $\pi\in\sn$ and suppose $(P,Q)=\RSK(\pi)$.
Then the reverse of $\pi$ corresponds to $(P\myt,\mathrm{evac}(Q)\myt)$ 
where $\myt$ denotes the reflection of a tableaux in the line $x+y=0$ and is also called the {\it{transpose}}; 
see {Stanley's book}~\cite[Corollary A1.2.11]{stanley1} for a discussion of $\mathrm{evac}(Q)$.
This means the shape of the reverse of $\pi$ will be the conjugate of $\shape(\pi)$.
Thus
\begin{align}
B_n(q,q) 
= & \sum_{\pi \in \sn} q^{\bump(\shape(\pi))+\bump((\shape(\pi))')} \nonumber \\
= & \sum_{\lambda \vdash n} (f^{\lambda})^2  q^{\bump(\lambda)+\bump(\lambda')}.\label{harvey}
\end{align}

For the proof of this proposition we will use the following result (that appears in Bergeron's book~\cite[p. 33]{bergeron}):
\begin{align}
\displaystyle\sum_{c \in \lambda} h(c) =  \bump(\lambda) + \bump(\lambda') + |\lambda|.\label{mudpile}
\end{align}
This sum of hook-lengths result follows from the following argument.
\newcommand{\leg}{\text{leg}}
\newcommand{\arm}{\text{arm}}
Let $\leg(c)$ denote the number of cells strictly below $c$, in the same column, and 
let $\arm(c)$ denote the number of cells strictly to the right of $c$, in the same row.
Then:
\begin{align*}
\sum_{c \in \lambda} h(c) &= \sum_{c \in \lambda}(1+\leg(c)+\arm(c))\\
    & = n + \displaystyle\sum_{c \in \lambda}\leg(c) + \sum_{c \in \lambda}\arm(c)\\
    & = |\lambda| + \sum_{i=1}^{\lambda_1'} \sum_{k=i+1}^{\lambda_1'}\lambda_k + \sum_{j=1}^{\lambda_1}\sum_{l=j+1}^{\lambda_1}\lambda_l'\\
    & = |\lambda| + \sum_{k=2}^{\lambda_1'}\sum_{i=1}^{k-1}\lambda_k + \sum_{l=2}^{\lambda_1}\sum_{j=1}^{l-1}\lambda_l'\\
    &= |\lambda| + \sum_{k=1}^{\lambda_1'}(k-1)\lambda_k + \sum_{l=1}^{\lambda_1}(l-1)\lambda_l'\\
    &= |\lambda| + \bump(\lambda) + \bump(\lambda').
\end{align*}
The last step follows from Proposition~\ref{bumpinshape}.
So we can use (\ref{mudpile}) to rewrite (\ref{harvey}) as:
\begin{align*}
B_n(q,q)= & \sum_{\lambda \vdash n} (f^{\lambda})^2  q^{-n+\sum_{c \in \lambda} h(c)}\\
= & \dfrac{1}{q^{n}} \sum_{\lambda \vdash n} (f^{\lambda})^2  \prod_{c \in \lambda} q^{h(c)}. \qedhere
\end{align*}
\end{proof}

Proposition~\ref{weakbumpprop} allows us to see that the generating function for the weakbump statistic over permutations is the reversal of the generating function for the longest increasing statistic over permutations. There is no known formula for this generating function, but it appears as sequence A047874 in the On-Line Encyclopedia of Integer Sequences~\cite{oeis}.

\subsection{The partition-bump polynomial}
There is more we can say about the generating function of the bump statistic over partitions rather than permutations.
This generating function is 
\begin{align}
T_n(q) =& \sum_{\lambda \vdash n} q^{\bump(\lambda)} = \sum_i t_{n,i} q^i,
\end{align}
where $t_{n,i}$ is the number of partitions $\lambda \vdash n$ with $\bump(\lambda)=i$.
The first few polynomials are listed in Table~\ref{firstfew}.

\begin{theorem}
$T_n(q) = [t^n] ~ \displaystyle \prod_{i=1}^{n} \dfrac{1}{1-q^{\binom{i}{2}}t^i }.$
\end{theorem}

\begin{proof}
Let $\Part$ be the set of all integer partitions $\lambda$ and let $\Part_a$ be those partitions $\lambda$ whose largest part $\lambda_1=a$.
Define the generating function
$$f(q,t,z) := \sum_{\lambda \in \Part\atop \lambda=(\lambda_1,\lambda_2,\ldots)} q^{\bump(\lambda)} t^{|\lambda|} z^{\lambda_1}.$$
Notice that every $\lambda \in \Part_{\ell}$ can be written uniquely as a pair $\unpair{\ell,\mu}$ where 
$\mu \in \Part_k$ for some $k \leq \ell$.
This decomposition allows us to write $\bump(\lambda) = (|\lambda|-\ell) +\bump(\mu) = |\mu|+\bump(\mu)$ and $|\lambda| = \ell+|\mu|$.
Using this we have
\begin{align*}
f(q,t,z) &= \sum_{\ell \geq 0} z^{\ell} \sum_{\lambda \in \Part_{\ell}} q^{\bump(\lambda)} t^{|\lambda|} \\
&= \sum_{\ell \geq 0} z^{\ell}  \sum_{\lambda = \unpair{\ell,\mu} \atop \mu \in \Part_k ~:~ k\leq \ell} q^{\bump(\lambda)} t^{|\lambda|} \\
&= \sum_{\ell \geq 0} z^{\ell} \sum_{k=0}^{\ell} \sum_{\mu \in \Part_k} q^{|\mu| +\bump(\mu)} t^{\ell+|\mu|} \\
&=\sum_{k \geq 0} \sum_{\mu \in \Part_k} q^{\bump(\mu)} (qt)^{|\mu|} 
\sum_{\ell \geq k} (tz)^{\ell} 
\end{align*}
The inner sum is $(tz)^k / (1-tz)$, and so 
\begin{align*}
f(q,t,z) 
&=\dfrac{1}{1-tz} \sum_{k \geq 0} (tz)^k \sum_{\mu \in \Part_k} q^{\bump(\mu)} (qt)^{|\mu|} 
=\dfrac{1}{1-tz} f(q,qt,tz).
\end{align*} 
Iterating this equation gives
$$f(q,t,z) = \dfrac{1}{1-tz} \dfrac{1}{1-qt^2z}  f(q,q^2t,qt^2z) =  \dfrac{1}{1-tz} \dfrac{1}{1-qt^2z} \dfrac{1}{1-q^3 t^3 z} f(q,q^3t,q^3 t^3 z) .$$ 
We can do this $n$ times {to obtain}:
$$f(q,t,z) = \left( \prod_{i=1}^{n} \dfrac{1}{\left(1-q^{\binom{i}{2}}t^i z\right)} \right) f(q,q^nt,q^{\binom{n}{2}}t^n z).$$
Setting $z=1$ gives us the following expression for $T_n(q)$:
\begin{align*}
 T_n(q) &= [t^n] ~ \prod_{i=1}^{n} \dfrac{1}{1-q^{\binom{i}{2}}t^i }.\qedhere
\end{align*}
\end{proof}

When one looks at the first dozen polynomials $T_n(q)$, it will become apparent that the head of each of these polynomials begins to remain constant as $n$ increases. 
This convergence of the head of the polynomial is made precise in the following theorem.

\begin{theorem}
For all $n,j \geq 0$, $t_{n+j,n} \leq t_{n+j+1,n}$.
Also, for $i\in\{0,1,\dots,\floor{ n/2} \}$, $$t_{n,i}=[z^{i}]\displaystyle\dfrac{1}{\displaystyle\prod_{k\geq 2}\left(1 - z^{\binom{k}{2}}\right)}.$$
\end{theorem}
    
Note that the coefficients of the first $\floor{ n/2}$ powers are independent of $n$.

\begin{proof}
Recall that for a shape $\lambda$, we have $\bump(\lambda)=\sum_{i \geq 2}(i-1)\lambda_{i}$ and the entries of $\lambda$ are weakly decreasing.
For a given $n$, suppose there exists a solution $\lambda = (\lambda_{1}=x_{1},\lambda_{2}=x_{2},\lambda_{3}=x_{3},\ldots) \vdash n$ 
to the equation $\sum_{i \geq 2}(i-1)\lambda_{i}=k$. 
Then, for $m > n$, we have $(\lambda_{1}=x_{1}+m-n,\lambda_{2}=x_{2},\ldots) \vdash m$ is also a solution, which clearly has the same bump statistic as $\lambda$.
Hence $t_{n,k} \leq t_{m,k}$ and we have proven the first claim in the theorem.

Secondly, if $\sum_{i \geq 2}(i-1)\lambda_{i}=k$ then $\sum_{i \geq 2}\lambda_{i} \leq k$ while $\sum_{i \geq 2}\lambda_{i}=k$ iff $\lambda_{2}=k$ (and $\lambda_{i}=0$ for all $i>2$). This shows that as $n$ increases we will find a new, and final, solution to $\bump(\lambda)=k$ precisely when $n=2k$ because with $n=2k$ we have $\lambda=(k,k,0,0,\dots) \vdash n$.

To complete the proof, we note that $\bump(\lambda)$ can be thought of as the sum of the number of times that each value in the cells of the final tableau was bumped during the RS process. For any column, $\lambda_{j}'$, the contribution to $\bump(\lambda)$ from the values in the cells of the column is, reading down the column, $0+1+2+\dots+(|\lambda_{j}'|-1)=\binom{|\lambda_{j}'|}{2}$, i.e. a triangular number. So each solution to $\bump(\lambda)=k$ is a sum of triangular numbers. And vice versa, if a sum of triangular numbers is equal to $k$, then in ordering those numbers to be weakly descending they correspond to a shape $\lambda$ that is a solution.
\end{proof}

\end{document}